\documentclass[11pt,a4paper,reqno]{amsart}
\usepackage[english]{babel}
\usepackage[applemac]{inputenc}
\usepackage[T1]{fontenc}
\usepackage{palatino}
\usepackage{bbm}
\usepackage{cite}
\usepackage{float, verbatim}
\usepackage{amsmath}
\usepackage{amssymb}
\usepackage{amsthm}
\usepackage{amsfonts}
\usepackage{graphicx}
\usepackage{mathtools}
\usepackage{enumitem}
\usepackage[colorlinks = true, citecolor = black]{hyperref}
\usepackage{color}
\usepackage{tikz}
\usepackage{xcolor}
\usepackage{pgfplots}
\usepackage{caption}
\usepackage{subcaption}
\usepackage{enumerate}

\newcommand{\ubox}{\overline{\dim_{\mathrm{B}}}}
\newcommand{\lbox}{\underline{\dim_{\mathrm{B}}}}
\newcommand{\nbox}{\dim_{\mathrm{B}}}

\newcommand{\Haus}{\dim_{\mathrm{H}}}

\newtheorem*{thm*}{Theorem}

\newtheorem{thm}{Theorem}[section]

\newtheorem{lma}[thm]{Lemma}

\newtheorem{defn}[thm]{Definition}

\newtheorem{conj}[thm]{Conjecture}
\newtheorem{rem}[thm]{Remark}

\newtheorem{exm}[thm]{Example}

\begin{document}
	
	\title[Bernoulli decomposition]{Bernoulli decomposition and arithmetical independence between sequences}
	
	\author{Han Yu}
	\address{Han Yu\\
		Department of Pure Mathematics and Mathematical Statistics\\University of Cambridge\\CB3 0WB \\ UK }
	\curraddr{}
	\email{hy351@maths.cam.ac.uk}
	\thanks{}
	
	\subjclass[2010]{28D20;11J71;28A80}
	
	\keywords{Independence of sequences, Bernoulli decomposition, disjointness between dynamical systems}
	
	\date{}
	
	\dedicatory{}
	
	\begin{abstract}
		In this paper, we study the following set\[A=\{p(n)+2^nd \mod 1: n\geq 1\}\subset [0,1],\]
		where $p$ is a polynomial with at least one irrational coefficient on non constant terms, $d$ is any real number and for $a\in [0,\infty)$, $a \mod 1$ is the fractional part of $a$. With the help of a method recently introduced by Wu, we show that the closure of $A$ must have full Hausdorff dimension.
	\end{abstract}
	
	\maketitle
	\allowdisplaybreaks
	\section{Introduction and background}
	In this paper, we follow a Bernoulli decomposition method developed in \cite{Wu}. This method combines Sinai's factor theorem with some properties of Bernoulli shifts and solves a dimension version of Furstenberg's intersection problem. Here, we will consider a very different number-theoretic problem with a similar method. Let $\alpha$ be an irrational number and we know that the sequence (irrational rotation orbit) $\{n\alpha\mod 1\}_{n\geq 1}$ equidistributes in $[0,1].$ Let $X_n,n\geq 1$ be a sequence of $i.i.d$ real-valued random variables. For convenience, let $X_1$ be uniformly distributed in $[0,1].$ In this setting, one can show that $\{n\alpha+X_n\mod 1\}_{n\geq 1}$ equidistributes almost surely and in particular its closure contains intervals. Now, we replace the random sequence $X_n$ with a deterministic sequence $\{2^n d\mod 1\}_{n\geq 1}$ by choosing an arbitrary real number $d.$ On one hand, if $d$ is `simple' enough, say, a rational number, then it is straightforward that $\overline{\{2^n d+n\alpha\mod 1\}_{n\geq 1}}$ contains intervals. On the other hand, if $d$ is `random' enough, say, chosen randomly according to the Lebesgue measure, then by simple probabilistic arguments one can show that almost surely, $\{2^n d+n\alpha\mod 1\}_{n\geq 1}$ again equidistributes and its closure contains intervals. This consideration leads us to the following conjecture.
	\begin{conj}\label{conj}
		Let $\alpha$ be an irrational number and $d$ be a real number. Then the topological closure of the sequence $\{2^n d+n\alpha\mod 1\}_{n\geq 1}$ contains intervals.
	\end{conj}  
	In this paper, we prove the following partial result towards the above conjecture.
	\begin{thm}
		Let $\alpha$ be an irrational number and $d$ be a real number. Then the topological closure of the sequence $\{2^n d+n\alpha\mod 1\}_{n\geq 1}$ has Hausdorff dimension $1$.
	\end{thm}
	In fact, we will prove a stronger result, Theorem \ref{main}. Before we state this theorem, we provide some more backgrounds. Given two sequences $x=\{x_n\}_{n\geq 1}, y=\{y_n\}_{n\geq 1}$ in $[0,1],$ it is often interesting to study their independence. In terms of sequences with dynamical backgrounds, this can be also understood as the disjointness between dynamical systems, see \cite{F67} for more details. Intuitively, we want to say that two sequences $x,y$ are independent if $\{(x_n,y_n)\}_{n\geq 1}$ is in some sense close to the product set $X\times Y,$ where $X,Y$ are the sets of numbers in the sequence $x,y$ respectively. We give a natural way of expressing this idea.
	\begin{defn}
		Let $x=\{x_n\}_{n\geq 1}, y=\{y_n\}_{n\geq 1}$ be two sequences in $[0,1].$ We write $X,Y$ to be the sets of numbers in the sequence $x,y$ respectively. Then we say that $x$ and $y$ are arithmetically independent if the set $H(x,y)$ of numbers in the sequence $\{x_n+y_n\}_{n\geq 1}$ attains the largest possible box dimension, namely,
		\[
		\lbox H(x,y)=\min\{1, \lbox X+\lbox Y\}.
		\]
	\end{defn}
	As an easy example, we see that $\{n\alpha\}_{n\geq 1}$ and $\{n\beta\}_{n\geq 1}$ are arithmetically independent if $1,\alpha,\beta$ are linearly independent over the field $\mathbb{Q}.$ It is also possible to study the independence between $\{n\alpha\}_{n\geq 1}$ and $\{n^2\beta\}_{n\geq 1}$ based on Weyl's equidistribution theorem. Naturally, a next question is to ask about the independence between $\{n\alpha\}_{n\geq 1}$ and $\{2^n d\}_{n\geq 1}$, where $d$ is any real number.  For a polynomial $p$ with degree $k$ with real coefficients, we write $p(n)=\sum_{i=0}^k a_i n^i.$ We say that $p$ is irrational if at least one of the numbers $a_1,\dots,a_k$ is an irrational number. In this paper, we show the following result. See Section \ref{Ab} for a clarification of the notations that appear below.
	
	\begin{thm}\label{main}
		Let $p$ be an irrational polynomial and let $d$ be any real number. Then the sequences $\{p(n) \mod 1\}_{n\geq 1}$ and $\{2^n d\mod 1\}_{n\geq 1}$ are arithmetically independent. In fact, we have the following stronger result \[\Haus \overline{\{p(n)+2^n d \mod 1\}_{n\geq 1}}=1.\]
	\end{thm}
	
	We note that there is a curious connection between sequences of form $\{p(n)+2^n d\mod 1\}_{n\geq 1}$ and  $\alpha\beta$-sequences. Let $\alpha,\beta$ be two real numbers, an $\alpha\beta$-sequence $\{x_n\}_{n\geq 1}$ is such that $x_1=0$ and for each $i\geq 1$ we can choose $x_{i+1}=x_i+\alpha\mod 1$ or $x_{i+1}=x_i+\beta\mod 1$ freely. We have the following problem.
	\begin{conj}\label{feng}
		Let $\alpha,\beta$ be such that $1,\alpha,\beta$ are independent over the field of rational numbers. Then any $\alpha\beta$-sequence has full box dimension.
	\end{conj} 
	This conjecture is related to affine embeddings between Cantor sets, symbolic dynamics and Diophantine approximation, see \cite{K79}, \cite{FX18} and \cite{Y18}. A lot of ideas for proving Theorem \ref{main} appeared in \cite{Y18} for $\alpha\beta$-sets. For this reason, we can consider Theorem \ref{main} as a cousin of Conjecture \ref{feng}. Although the method in this paper cannot be used directly for $\alpha\beta$-sequences, it still sheds some lights on Conjecture \ref{feng}. However, at this stage, we mention that in \cite{K79} there is a construction of an $\alpha\beta$-sequence whose closure does not have full Hausdorff dimension. 
	
	We also consider here a number-theoretic result which is closely related to what has been discussed. Let $m$ be an odd number. We consider the ring $R[m]$ of residues modulo $m.$ It is the finite set $\{0,\dots,m-1\}$ together with the integer multiplication and addition modulo $m$. In this setting, we can also consider the sequence $\{2^n+cn\mod m\}_{n\geq 0}$ where $c$ is an integer such that $gcd(c,m)=1.$ On one hand, the $+c \mod m$ action on $R[m]$ can be seen as uniquely ergodic, which is analogous to $+\alpha\mod 1$ action on the unit interval with an irrational number $\alpha.$ On the other hand, $\{2^n \mod m\}_{n\geq 0}$ is an orbit under the $\times 2\mod m$ action. An analogy of Theorem \ref{main} would be that $\{2^n+cn\mod m\}_{n\geq 0}$ is large in $R[m].$ We show the following result which confirms this intuition. We remark that the method for proving the following result shares some strategies for proving Theorem \ref{main}. 
	\begin{thm}\label{num}
		Let $m\geq 3$ be an odd number and $c$ be such that $gcd(c,m)=1$. Let $D(m)$ be the number of residue classes visited by $\{2^n+cn\mod m\}_{n\geq 0}.$ Then $D(m)=m.$ In other words, for each $r\in R[m],$ there is an integer $n_r$ such that $2^{n_r}+cn_r\equiv r\mod m.$
	\end{thm}

The above result is a special case of Problem 6 in the third round of the 27-th Brazilian Mathematical Olympiad, see \cite{MO}.

	\section{definitions and notations}
	\subsection{Logarithm} We make the convention that the $\log$ function has base $2.$
	
	\subsection{Dimensions}
	We list here some basic definitions of dimensions mentioned in the introduction.  For more details, see \cite[Chapters 2,3]{Fa} and \cite[Chapters 4,5]{Ma1}. We shall use $N(F,r)$ for the minimal covering number of a set $F$ in $\mathbb{R}^n$ with closed balls of side length $r>0$. 
	
	\subsubsection{Hausdorff dimension}
	
	Let $g: [0,1)\to [0,\infty)$ be a continuous function such that $g(0)=0$. Then for all $\delta>0$ we define the following quantity
	\[
	\mathcal{H}^g_\delta(F)=\inf\left\{\sum_{i=1}^{\infty}g(\mathrm{diam} (U_i)): \bigcup_i U_i\supset F, \mathrm{diam}(U_i)<\delta\right\}.
	\]
	The $g$-Hausdorff measure of $F$ is
	\[
	\mathcal{H}^g(F)=\lim_{\delta\to 0} \mathcal{H}^g_{\delta}(F).
	\]
	When $g(x)=x^s$ then $\mathcal{H}^g=\mathcal{H}^s$ is the $s$-Hausdorff measure and Hausdorff dimension of $F$ is
	\[
	\Haus F=\inf\{s\geq 0:\mathcal{H}^s(F)=0\}=\sup\{s\geq 0: \mathcal{H}^s(F)=\infty          \}.
	\]
	\subsubsection{Box dimensions}
	The upper box dimension of a bounded set $F$ is
	\[
	\overline{\nbox} F=\limsup_{r\to 0} \left(-\frac{\log N(F,r)}{\log r}\right).
	\]
	Similarly the lower box dimension of $F$ is
	\[
	\lbox F=\liminf_{r\to 0} \left(-\frac{\log N(F,r)}{\log r}\right).
	\]
	If the limsup and liminf are equal, we call this value the box dimension of $F$ and we denote it as $\nbox F.$
	\subsection{The unconventional fractional part symbol}\label{Ab}
	For a real number $\alpha$, it is conventional to use $\{\alpha\}$ for its fractional part. It is unfortunate that $\{.\}$ is also used to denote a set or a sequence as well. For this reason we will use $\mod 1$ for the fractional part. More precisely, for a real number $x$ we write $x\mod 1$ to denote the unique number $a$ in $[0,1)$ such that $a-x$ is an integer.
	\subsection{Sets and sequences}
	We write $\{x_n\}_{n\geq 1}$ for the sequence $x_1x_2x_3\dots.$ Sometimes it is convenient to use $\{x_n\}_{n\geq 1}$ to denote the following set
	\[
	\{x: \exists n\in\mathbb{N}, x=x_n\}.
	\]
	Thus $\overline{\{x_n\}_{n\geq 1}}$ and $\lbox \{x_n\}_{n\geq 1}$ should be understood in this way.
	\subsection{Filtrations, atoms and entropy}\label{ATOM}
	Let $X$ be a set with $\sigma$-algebra $\mathcal{X}.$ A filtration of $\sigma$-algebras is a sequence $\mathcal{F}_n\subset\mathcal{X},n\geq 1$ such that
	\[
	\mathcal{F}_1\subset \mathcal{F}_2\subset\dots \subset \mathcal{X}.
	\]
	Given a measurable map $S:X\to X$ and a finite measurable partition $\mathcal{A}$ of $X$, we denote $S^{-n} \mathcal{A}$ to be the following finite collection of sets (notice that $S$ might not be invertible)
	\[
	\{S^{-n}(A) :A\in \mathcal{A}\}.
	\]
	Then we use $\vee_{i=0}^{n-1} S^{-i}\mathcal{A}$ to be the $\sigma$-algebra generated by $S^{-i} \mathcal{A},i\in [0,n-1].$ An atom in $\vee_{i=0}^{n-1} S^{-i}\mathcal{A}$ is a set $A$ that can be written as
	\[
	A=\bigcap_{i} C_i
	\]
	where for each $i\in\{0,\dots,n-1\}$,  $C_i\in S^{-i}\mathcal{A}$.
	In this sense $\vee_{i=0}^{n-1} S^{-i}\mathcal{A}$ is generated by a finite partition $\mathcal{A}_{n-1}$ of $X$ which is finer than $\mathcal{A}.$ Let $\mu$ be a probability measure, then we define the Shannon entropy of $\mu$ with respect to a finite partition $\mathcal{A}$ as follows
	\[
	H(\mu,\mathcal{A})=-\sum_{A\in\mathcal{A}} \mu(A)\log \mu(A).
	\]
	We define the entropy of $S$ as follows
	\[
	h(S,\mu)=\lim_{n\to\infty} \frac{1}{n} H(\mu, \mathcal{A}_{n-1}),
	\]
	where $\mathcal{A}$ is a partition such that $\vee_{i=1}^\infty S^{-i}\mathcal{A}=\mathcal{X}.$ Here we implicitly assumed that such a generating partition exists and used Sinai's entropy theorem, see \cite[Lemma 8.8]{PY}.
	
	Let $\mathcal{Y}\subset \mathcal{X}$ be an $S$-invariant $\sigma$-algebra, i.e. $S^{-1}(\mathcal{Y})\subset \mathcal{Y}.$ Let $n\geq 1$ be an integer. We define the conditional information function of $\mathcal{A}_n$ conditioned on $\mathcal{Y}$ as follows,
	\[
	I_{\mu,\mathcal{A}_n|\mathcal{Y}}(x)=-\log E_\mu[\mathbbm{1}_{A_n(x)}|\mathcal{Y}](x).
	\]
	Here, $A_n(x)$ is the atom of $\mathcal{A}_n$ which contains $x\in X.$ Then, we define the conditional Shannon entropy of $\mathcal{A}_n$ conditioned on $\mathcal{Y}$ as
	\[
	H(\mu, \mathcal{A}_n|\mathcal{Y})=\int I_{\mu,\mathcal{A}_n|\mathcal{Y}}(x)d\mu(x).
	\]
	Finally, we define the conditional entropy of $S$ conditioned on $\mathcal{Y}$ as
	\[
	h(S|\mathcal{Y},\mu)=\lim_{n\to\infty} \frac{1}{n} H(\mu, \mathcal{A}_{n-1}|\mathcal{Y}).
	\]
	All the above quantities are well defined, see \cite[Chapters 1,2]{D11} for more details.
	\subsection{Factors}\label{FACTOR}
	A measurable dynamical system is in general denoted as $(X,\mathcal{X},S,\mu)$ where $X$ is a set with $\sigma$-algebra $\mathcal{X}$, a measure $\mu$ (in this paper, $\mu$ will be a probability measure) and a measurable map $S:X\to X.$ In case when $\mathcal{X}$ is clear in context we do not explicitly write it down. Given two dynamical systems $(X,\mathcal{X},S,\mu)$, $(X_1,\mathcal{X}_1,S_1,\mu_1)$, a measurable map $f:X\to X_1$ is called a factorization map and $(X_1,\mathcal{X}_1,S_1,\mu_1)$ is called a factor of $(X,\mathcal{X},S,\mu)$ if $\mu_1=f\mu$ and $f\circ S(x)=S_1\circ f(x)$ holds for $\mu$ almost all $x\in X.$ 
	
	Another way of viewing factors is via invariant sub $\sigma$-algebras. Let $\mathcal{Y}\subset\mathcal{X}$ be a sub-$\sigma$-algebra which is invariant under the map $S.$ Then $(X,\mathcal{Y},S,\mu)$ can be seen as a factor of $(X,\mathcal{X},S,\mu)$ via the identity map. We can take $\mathcal{Y}=f^{-1}(\mathcal{X}_1)$ in the previous paragraph. In this measure theoretical sense, $(X_1,\mathcal{X}_1,S_1,\mu_1)$ and $(X,\mathcal{Y},S,\mu)$ can be viewed as the same dynamical system.
	\subsection{Bernoulli system}
	Let $\Lambda$ be a finite set of symbols and let $\Omega=\Lambda^{\mathbb{N}}$ be the space of one sided infinite sequences over $\Lambda.$ We define $S$ to be the shift operator, namely, for $\omega=\omega_1\omega_2\dots\in\Omega,$
	\[
	S(\omega)=\omega_2\omega_3\dots.
	\]
	We take the $\sigma$-algebra on $\Omega$ generated by cylinder subsets. A cylinder subset $Z\subset\Omega$ is such that
	$
	Z=\prod_{i\in\mathbb{N}}Z_i
	$
	and $Z_i=\Lambda$ for all but finitely many integers $i\in\mathbb{N}.$ We construct a probability measure $\mu$ on $\Omega$ by giving a probability measure $\mu_{\Lambda}=\{p_\lambda\}_{\lambda\in \Lambda}$ on $\Lambda$ and set $\mu=\mu^{\mathbb{N}}_{\Lambda}.$ We require here that $p_\lambda\neq 0$ for all $\lambda\in \Lambda$. Then this system is weak-mixing and has entropy $h(S,\mu)=\sum_{\lambda\in\Lambda} -p_\lambda\log p_\lambda.$ We call this system a Bernoulli system. 
	
	\begin{comment}
	We can also introduce a metric topology on $\Omega$ by defining $d(\omega,\omega')=\#\Lambda^{-\min\{i\in\mathbb{N}: \omega_i\neq\omega'_i\}}.$ This turns $\Omega$ into a compact and totally disconnected space. For $\omega\in\Omega$ and $r\in (0,1)$, we use $B(\omega,r)$ to denote the $r$-ball around $\omega$ with radius $r$ with respect to the metric $d$ constructed above.
	\end{comment}
	\subsection{Joinings}
	Let $(X,\mathcal{X},S,\mu)$ and $(Y,\mathcal{Y},T,\nu)$ be two measurable dynamical systems. A joining between those two dynamical systems is an $S\times T$ invariant probability measure $\rho$ on $X\times Y$ (with respect to the product $\sigma$-algebra $\sigma(\mathcal{X}\times \mathcal{Y})$) such that $\pi_X \rho=\mu, \pi_Y \rho=\nu.$ The two systems $(X,\mathcal{X},S,\mu)$ and $(Y,\mathcal{Y},T,\nu)$ are \emph{disjoint} if the only joining is the product measure $\mu\times \nu.$ The follow example can be found in \cite[Theorem I.4]{F67}.
	
	\begin{exm}\label{EXM}
		Let $(X,\mathcal{X},S,\mu)$ be a measure theoretically
		distal ergodic system with finite height. Let $(Y,\mathcal{Y},T,\nu)$ be a weakly mixing system. Then $(X,\mathcal{X},S,\mu)$ and $(Y,\mathcal{Y},T,\nu)$ are disjoint.
	\end{exm}
	
	A measure theoretically distal ergodic system with finite height is obtained from a Kronecker system with finitely many ergodic group extensions. For example, irrational rotations on $\mathbb{T}=\mathbb{R}/\mathbb{Z}$ with the Lebesgue measure are Kronecker systems. The transformation $(x,y)\in\mathbb{T}^2\to (x+\alpha,x+y)$ on $\mathbb{T}^2$ with $\alpha\notin \mathbb{Q}$ is obtained from an irrational rotation with an ergodic group extension. In this paper, we will also consider the  transformation $(x_1,\dots,x_n)\in\mathbb{T}^n\to (x_1+\alpha,x_2+x_1,x_3+x_2,\dots,x_n+x_{n-1})$ on $\mathbb{T}^n.$ The above are examples of measure theoretically distal ergodic systems with finite height.

	\section{A mathematical Olympiad problem}
	We first illustrate a short proof of Theorem \ref{num}, which provides us with some motivation.
	\begin{proof}[Proof of Theorem \ref{num}]
		Let $l=ord(2,m)$ be the order of $2$ in the multiplication group $(\mathbb{Z}/m\mathbb{Z})^*.$ This can be done because $gcd(2,m)=1.$ For convenience, we consider $c=1$ and note that other cases can be shown with the same method. Since $l=ord(2,m)$ we consider the following sequence
		\[
		\{2^{nl}+nl\mod m\}_{n\geq 0}.
		\]
		We see that $2^{nl}\equiv 1\mod m$ for all $n\geq 0.$ However $H=\{nl \mod m\}_{n\geq 0}$ is a subgroup of $\mathbb{Z}/m\mathbb{Z}$ of order $m/gcd(l,m).$ For convenience we write $\Delta=gcd(l,m).$ This $\Delta$ plays the same role of the entropy in the proof of Theorem \ref{WU} which leads to Theorem \ref{main}. If $\Delta=1$ then $D(m)=m$ follows automatically. We consider the case when $\Delta>1.$ Now for each integer $r$ we consider the following sequence
		\[
		\{2^{r+nl}+r+nl\mod m\}.
		\]
		This sequence forms a coset of $H.$ More precisely it is $2^r+r+H.$ Now if $\{2^r+r\mod \Delta\}_{r\geq 0}$ would visit all residue classes modulo $\Delta,$ then $2^r+r+H,r\geq 0$ would visit all cosets of $H$ in $\mathbb{Z}/m\mathbb{Z}$ and $\{2^n+n\}_{n\geq 1}$ would visit all residue classes modulo $m.$
		Since $\Delta$ is an odd number as well we see that we have reduced the problem for $m$ to the problem for $\Delta$ which is strictly smaller than $m.$ We can iterate this reduction procedure. Since we are considering positive integer set, either we eventually obtain $\Delta=1$ or else we can consider further $gcd(\Delta,ord(2,\Delta))<\Delta.$ The latter can not happen infinitely often. This concludes the proof.
	\end{proof}
	
	\section{A consequence of Sinai's factor theorem}
	In this section, we discuss a consequence of Sinai's factor theorem. As mentioned in the introduction, this section is strongly influenced by \cite[Section 6]{Wu}. To some extent, the idea resembles the arguments in the previous section. We start this section by introducing the set-ups and making some standard considerations.
	
	Let $(X,\mathcal{X},S,\mu)$ be a measure theoretically distal ergodic system with finite height. Here we assume that $\mu$ is a probability measure on the $\sigma$-algebra $\mathcal{X}$. Let $(Y,\mathcal{Y},T,\nu)$ be an ergodic measurable dynamical system. Furthermore, we require that $T$ admits a finite generator, i.e. a finite measurable partition $\mathcal{A}_0$ of $Y$ such that $\vee_{i=0}^{\infty} T^{-i}\mathcal{A}_0$ is $\mathcal{Y}$. For convenience, we put the following definition.
	
	\begin{defn}\label{BALL}
		Let $(Y,T,\nu), \mathcal{A}_0$ be as given in above. Let $B\subset Y.$ For each integer $n\geq 1,$ we define $N_{\mathcal{A}_0,S,n}(B)$ to be the number of atoms in $\mathcal{A}_n$ intersecting $B.$ Then we define the following quantities:
		\[
		\overline{\dim_{\mathcal{A}_0,S}} B=\limsup_{n\to\infty}\frac{\log N_{\mathcal{A}_0,S,n}(B)}{n}.
		\]
		\[
		\underline{\dim_{\mathcal{A}_0,S}} B=\liminf_{n\to\infty} \frac{\log N_{\mathcal{A}_0,S,n}(B)}{n}.
		\]
	\end{defn}
	
	For example, given $\lambda>0$, if $Y\subset\mathbb{R}$ and $diam(A_n(x))= O(2^{-\lambda n})$ uniformly for all $n,x$ then 
	\[
	N(B,2^{-\lambda n})=O(N_{\mathcal{A}_0,S,n}(B)).
	\]
	In this case, if $    \overline{\dim_{\mathcal{A}_0,S}} B=0$ then $\ubox B=0.$ The main goal of this section is to show the following result which is a variant of Wu's ergodic theoretic result in \cite[Section 6]{Wu}.
	
	\begin{thm}\label{WU}\footnote{Later on, we only use this result with $X,Y$ being compact metric spaces with Borel $\sigma$-algebras and $\dim_{\mathcal{A}_0,S}$ is equivalent to the box counting dimension on $Y.$}
		Let $(X,S,\mu),(Y,T,\nu)$ be as stated in above. Let $\rho$ be a joining between those two systems. Then $\rho$ admits a $\sigma(\mathcal{X}\times\mathcal{Y})$-measurable measure disintegration
		\[
		\rho=\int _{\Omega}\rho_\omega d\omega,
		\]
		where $(\Omega,d\omega)$ is a probability space such that for each $\epsilon>0$, there is a set $E$ with positive $d\omega$ measure and for $\omega\in E$,
		\begin{itemize}
			\item $\pi_X \rho_\omega=\mu.$
			\item There is a $\mathcal{Y}$-measurable set $B_\omega\subset Y$ such that $    \overline{\dim_{\mathcal{A}_0,S}} B_\omega\leq \epsilon$ and $\rho_\omega(\pi^{-1}_Y(B_\omega))>0.$
		\end{itemize}
		
	\end{thm}
	The proof of this theorem will be divided into two parts. Our first step is as follows.
	\subsection{Step One: The conditional Shannon-McMillan-Breiman theorem and a counting argument}
	\begin{lma}\label{CSMB}
		Let $(Y,T,\nu),\mathcal{A}_0$ be as stated in the beginning of this section. Let $\mathcal{B}$ be a countably generated $T$-invariant sub $\sigma$-algebra of $\mathcal{Y}.$ Suppose that the conditional entropy $h(T|\mathcal{B},\nu)=0.$ Then for $\nu.a.e$ $y\in Y$ and all $\epsilon>0$, there is a $\mathcal{Y}$-measurable set $B_{y,\epsilon}$ with     $\overline{\dim_{\mathcal{A}_0,S}} B_{y,\epsilon}\leq \epsilon.$ Moreover,for each $\epsilon>0,$ there is a $\mathcal{B}$-measurable set $E$ with positive $\nu$ measure and $\nu^{\mathcal{B}}_y(B_{y,\epsilon})>0$ for $y\in E.$ 
	\end{lma}
	\begin{proof}
		The conditional Shannon-McMillan-Breiman theorem (see \cite[Appendix B]{D11}) implies that for $\nu$ almost all $y\in Y$
		\[
		\lim_{n\to\infty} \frac{1}{n} I_{\nu,\mathcal{A}_n|\mathcal{B}}(y)=h(T|\mathcal{B},\nu).
		\]
		Let $\epsilon>0$ be a small number. Let $k\geq 0$ be an integer and we construct the following set
		\[
		B_{k}=\{y\in Y: \forall n\geq k, I_{\nu,A_n|\mathcal{B}}(y)\leq n (h(T|\mathcal{B},\nu)+\epsilon) \}.
		\]
		Then we have $\nu(\cup_{k\geq 1} B_{k})=1$ and thus there is an integer $n_0>0$ such that $B_{n_0}$ has positive $\nu$ measure. We can choose $n_0$ to be sufficiently large to ensure that $\nu(B_{n_0})$ is very close to one. However, positivity here is enough for later use. 
		
		Suppose that $\nu=\int \nu^{\mathcal{B}}_y d\nu(y)$ is the measure disintegration of $\nu$ against the factor $\mathcal{B},$ see \cite[Theorem 5.14]{EW}(system of conditional measures). Then we see that for $\nu.a.e$ $y\in Y$
		\[
		E_\nu[\mathbbm{1}_{A_n(y)}|\mathcal{B}](y)=\nu^{\mathcal{B}}_y(A_n(y)). 
		\]
		Thus we have
		\[
		B_{n_0}=\{y\in Y: \forall n\geq n_0, \log \nu^{\mathcal{B}}_y(A_n(y))\geq -n (h(T|\mathcal{B},\nu)+\epsilon) \}.
		\]
		Let $A_n$ be an atom in $\mathcal{A}_n$ intersecting $B_{n_0}$ with $n\geq n_0.$ Then we see that for $\nu.a.e. y\in A_n\cap B_{n_0}$ we have
		\[
		\nu^{\mathcal{B}}_y(A_n)=\nu^{\mathcal{B}}_y(A_n(y))\geq 2^{-n (h(T|\mathcal{B},\nu)+\epsilon)}.
		\] 
		Those $\nu.a.e.$ choices of $y$ form a $\mathcal{B}$-measurable set. Thus, by dropping out a $\mathcal{B}$-measurable set with zero $\nu$ measure we can assume that the above holds whenever $y\in A_n\cap B_{n_0}.$

		Since $\mathcal{B}$ is countably generated, we see that the fibre $[y]_{\mathcal{B}}=\bigcap_{F\in\mathcal{B},y\in F} F$ is well-defined and $\mathcal{B}$ measurable. For $\nu.a.e.$ $y\in Y$ the measure $\nu^{\mathcal{B}}_y$ is in fact a well defined probability measure supported on $[y]_{\mathcal{B}}$ and this measure is determined by the atom $[y]$ (see \cite[Theorem 5.14(2)]{EW}). In what follows, we fix arbitrarily such a $y\in Y.$ Suppose that $A_n$ is an atom in $\mathcal{A}_n$ intersecting $B_{n_0}.$ Then by the argument in above, we see that if $A_n\cap [y]_{\mathcal{B}}\cap B_{n_0}\neq \emptyset,$
		\[
		\nu^{\mathcal{B}}_y(A_n)\geq  2^{-n (h(T|\mathcal{B},\nu)+\epsilon)}.
		\]
		This implies that the number of atoms in $\mathcal{A}_n$ intersecting $[y]_{\mathcal{B}}\cap B_{n_0}$ is at most
		\[
		2^{n (h(T|\mathcal{B},\nu)+\epsilon)}.
		\]
		We note that the above arguments hold for a set of $\nu.a.e$ $y\in Y.$ Since we have $h(T|\mathcal{B},\nu)=0$, there is an integer $n_0\geq 1$ such that for $\nu.a.e.$ $y\in Y,$ all $n\geq n_0,$
		\[
		N_{\mathcal{A}_0,T,n}(B_{n_0}\cap [y]_{\mathcal{B}})\leq 2^{n \epsilon}.
		\]
		Thus $\overline{\dim_{\mathcal{A}_0,T}} B_{n_0}\cap [y]_{\mathcal{B}}\leq \epsilon.$ Moreover, we have $\nu(B_{n_0})>0,$ therefore we see that there is a $\mathcal{B}$-measurable set $E$ with  positive $\nu$ measure such that for $y\in E,$
		\[
		\nu^{\mathcal{B}}_y(B_{n_0}\cap [y]_{\mathcal{B}})>0.
		\]
		Note that $B_{n_0}\cap [y]_{\mathcal{B}}$ is $\mathcal{Y}$-measurable but not necessarily $\mathcal{B}$-measurable. This is the set $B_{y,\epsilon}$ as required.
	\end{proof}
	\subsection{Bernouli factors: Ornstein-Weiss's unilateral Sinai's factor theorem}
	For the second step, we need to use the unilateral Sinai's factor theorem which was proved in \cite{OW75}. Let $h=h(T,\nu)$ be the dynamical entropy of $(Y,T,\nu).$ Suppose that $h>0,$ then the unilateral Sinai's factor theorem says that any Bernoulli system $(\Omega,S_B,\nu_B)$ with entropy at most $h$ is a factor of $(Y,T,\nu).$ In particular, we can find a Bernoulli system as a factor of $(Y,T,\nu)$ with entropy $h.$ 
	\subsection{Step Two: Wu's ergodic theoretic result revisited}
	\begin{proof}[Proof of Theorem \ref{WU}]
		First, suppose that $h=h(T,\nu)=0.$ In this case we will see that the trivial disintegration $\rho=\rho$ works. Indeed, we have $\pi_X \rho=\mu, \pi_Y \rho=\nu$ since $\rho$ is a joining. As $h=0$, we see, by Lemma \ref{CSMB} with $\mathcal{B}$ being the trivial $\sigma$-algebra, that for each $\epsilon>0,$ there is a Borel set $B$ with positive $\nu$ measure such that
		\[
		\overline{\dim_{\mathcal{A}_0,T}} B\leq \epsilon.
		\]
		Then we see that $\rho(\pi^{-1}_Y(B))=\nu(B)>0.$ This finishes the proof in the case when $h=0.$
		
		Now suppose that $h>0.$ In this case, let $(\Omega,S_B,\mu_B)$ be a Bernoulli factor of $(Y,T,\nu)$ with entropy $h.$ This Bernoulli factor can be viewed as a $T$-invariant sub $\sigma$-algebra $\mathcal{B}$ in view of Section \ref{FACTOR}. This $\sigma$-algebra $\mathcal{B}$ is countably generated. Then we see that $\mathcal{C}=\pi^{-1}_Y(\mathcal{B})$ is a $S\times T$-invariant sub $\sigma$-algebra. Then we have the system of conditional measures $\rho^{\mathcal{C}}_{(x,y)}$ which are probability measures for $\rho.a.e. (x,y)\in X\times Y.$ Essentially, $\rho^{\mathcal{C}}_{(x,y)}$ does not depend on the choice of $x$. More precisely, we see that $[(x,y)]_{\mathcal{C}}=X\times [y]_{\mathcal{B}}.$ 
		
		By construction, $\pi_Y(\rho^{\mathcal{C}}_{(x,y)})=\nu^{\mathcal{B}}_y$ for $\rho.a.e.$ $(x,y),$ or equivalently, for $\nu.a.e.$ $y\in Y.$ Since $\mathcal{B}$ is obtained via a Bernoulli factor with entropy $h,$ we see that $h(T|\mathcal{B},\nu)=0$ (Abramov-Rokhlin formula \cite[Fact 4.1.6]{D11}). Then for $\nu.a.e.$ $y\in Y$ and all $\epsilon>0,$ we see from Lemma \ref{CSMB} that there is a $\mathcal{Y}$-measurable set $B_{y,\epsilon}$ (which could be empty) with
		\[
		\overline{\dim_{\mathcal{A}_0,T}} B_{y,\epsilon}\leq \epsilon.
		\]
		Moreover, for each $\epsilon>0,$ for a $\mathcal{B}$-measurable set $E$ with positive $\nu$ measure we have
		\[
		\nu^{\mathcal{B}}_y(B_{y,\epsilon})>0
		\]
		whenever $y\in E.$
		
		Let us take a measure $\rho^{\mathcal{C}}_{(x,y)}$ by taking a point $(x,y)$ (where $\rho^{\mathcal{C}}_{(x,y)}$ is defined as a probability measure) such that $y\in E$ and
		\[
		\rho^{\mathcal{C}}_{(x,y)}(\pi^{-1}_Y(B_{y,\epsilon}))=    \nu^{\mathcal{B}}_y(B_{y,\epsilon})>0.
		\] 
		Such choices of $(x,y)$ form a $\mathcal{C}$-measurable set $E'$ with positive $\rho$ measure. In order to finish the proof, we need to show that $\pi_X \rho^{\mathcal{C}}_{(x,y)}=\mu.$ To check this, let $f$ be a continuous function from $X$ to $\mathbb{R}.$ Then we see that by possibly dropping a $\mathcal{C}$-measurable $\rho$-null subset from $E'$,
		\[
		\int f(x') d\pi_X\rho^{\mathcal{C}}_{(x,y)}(x')=\int f(x')d\rho^{\mathcal{C}}_{(x,y)}(x',y')=E_\rho[f|\mathcal{C}](x,y)
		\]
		for $(x,y)\in E'.$ Observe that $\rho$ is $S\times T$-invariant. By construction, $(Y,\mathcal{B},T,\nu)$ is in fact a Bernoulli system. Observe that $\rho$ is also a joining between $(X,S,\mu)$ and $(Y,\mathcal{B},T,\nu)$. As Bernoulli system is weakly mixing, by Example \ref{EXM}, we see that $\rho$ must be equal to $\mu\times \nu$ viewed as a probability measure on the product $\sigma$-algebra $\sigma(\mathcal{X}\times\mathcal{B}).$ Since $\mathcal{C}=\pi^{-1}_Y(\mathcal{B})$ and $f$ is a function on $X,$ we see that for $(x,y)\in E',$
		\[
		E_\rho[f|\mathcal{C}](x,y)=\int fd\mu.
		\]
		As the above holds for all continuous functions on $X,$ we see that $\pi_X \rho^{\mathcal{C}}_{(x,y)}=\mu$ for $(x,y)\in E'.$ In other words, we have shown that $\rho=\int \rho^{\mathcal{C}}_{(x,y)}d\rho(x,y)$ is a measure disintegration satisfying the  statements of this theorem.
	\end{proof}
	
	\section{On sequences $\{p(n)+2^n d \mod 1\}_{n\geq 1}$}
	Now we prove Theorem \ref{main}.
	\begin{proof}[Proof of Theorem \ref{main}]
		First, let $\alpha\in (0,1)$ be an irrational number. We consider the sequence $\{n\alpha+2^n d\}.$ Consider the topological dynamical system $(\mathbb{T}\times\mathbb{T},S=R_\alpha\times T_2)$ where $R_\alpha$ is the $+\alpha\mod 1$ map and $T_2$ is the doubling map: $T_2(x)=2x\mod 1.$ Let $Z=\overline{\{S^n(0,d)\}_{n\geq 0}}.$ As $S$ is continuous, by Bogoliubov-Krylov theorem and ergodic decomposition, we can find an $S$-ergodic probability measure $\rho$ supported on $Z.$ Let $\mathcal{M}$ be the Borel $\sigma$-algebra on $\mathbb{T}$. Then we see that $\rho$ is a joining between $(\mathbb{T},\mathcal{M},R_\alpha,\mu)$ and $(\mathbb{T},\mathcal{M},T_2,\nu)$ where $\mu=\pi_1 \rho, \nu=\pi_2 \rho.$ Note that $\mu$ is the Lebesgue measure.
		
		Now we use Theorem \ref{WU}. For each $\epsilon>0,$ we can find a probability measure $\rho'$ supported on $Z$ such that $\pi_1 \rho'$ is the Lebesgue measure on $\mathbb{T}$ and there is a Borel set $B_\epsilon$ such that $\ubox B_\epsilon\leq \epsilon$ and $\rho'(\pi^{-1}_2(B_\epsilon))>0.$ Here, we choose $\mathcal{A}_0=\{[0,0.5),[0.5,1)\}$ for the doubling map. For this choice, we see that $\mathcal{A}_n$ consists dyadic intervals of length $2^{-n-1}.$ Then it is possible to see that $\overline{\dim_{\mathcal{A}_0,T_2}}$ coincides with the upper box dimension. Consider $A=\pi^{-1}_2(B_\epsilon)\cap Z.$ As $\rho'$ supports on $Z$, we see that
		\[
		\rho'(A)>0.
		\]
		Since $A$ is Borel, we see that $\pi_1(A)$ is Lebesgue measurable. However, as $\pi_1(A)$ might not be Borel measurable, we cannot use the fact that $\pi_1\rho'=\mu$ to deduce that $\pi_1(A)$ has positive Lebesgue measure since all measures here are only defined on Borel sets. If $\pi_1(A)$ has zero Lebesgue measure, then as it is Lebesgue measurable, we see that for each $\delta>0,$ we can cover $\pi_1(A)$ with open intervals with total length at most $\delta.$ Denote the union of those intervals as $A^\delta.$ Then $\pi^{-1}_1(A^\delta)$ is Borel and we have $\rho'(\pi^{-1}_1(A^\delta))=\mu(A^{\delta})\leq\delta.$ However, as $A\subset \pi^{-1}_1(A^\delta)$, we see that $\delta$ cannot be chosen arbitrarily small. Therefore $\pi_1(A)$ has positive Lebesgue measure and hence full Hausdorff dimension. Let $\Sigma$ denote the arithmetic sum map, i.e. $\Sigma(x,y)=x+y$ for $(x,y)\in\mathbb{T}\times\mathbb{T}.$ We have
		\[
		1=\Haus(\pi_1(A))\leq \Haus(\Sigma(A)-\pi_2(A))\leq \Haus(\Sigma(A)\times \pi_2(A))\leq \Haus(\Sigma(A))+\ubox \pi_2(A).
		\]
		Here we have used the fact that
		\[
		\pi_1(A)\subset \Sigma(A)-\pi_2(A)=\{a-b: (a,b)\in \Sigma(A)\times \pi_2(A)\}.
		\]
		We also used the fact that $\Sigma$ is a Lipschitz map. The rightmost inequality is a standard result in geometric measure theory, see \cite[Theorem 8.10]{Ma1}. Thus we see that
		\[
		\Haus \overline{\{n\alpha+2^n d\mod 1\}_{n\geq 0}}=\Haus \Sigma(Z)\geq\Haus \Sigma(A)\geq 1-\ubox \pi_2(A)\geq 1-\epsilon.
		\]
		As the above holds for all $\epsilon>0$ we see that $\Haus \overline{\{n\alpha+2^n d\mod 1\}_{n\geq 0}}=1.$
		
		Now we let $p$ be a polynomial with at least one irrational coefficient. Then the argument above for the special case $p(n)=n\alpha$ can be used here. We need to choose the $X$ component in Theorem \ref{WU} to be the transformation
		\[(x_1,\dots,x_n)\in\mathbb{T}^n\to (x_1+\alpha,x_2+x_1,x_3+x_2,\dots,x_n+x_{n-1})\] on $\mathbb{T}^n$ with a suitably chosen number $\alpha$ and $\Sigma$ to be the map:
		 \[(x_1,\dots,x_n,y)\to \Sigma(x_1,\dots,x_n,y)=x_n+y.\] See also \cite[Theorem 1.4]{EW} and its proof therein.
	\end{proof}
	
	\begin{rem}
		In fact, the above proof shows that for any non-empty closed $R_\alpha\times T_2$ invariant set $Z$, $\Sigma(Z)$ has full Hausdorff dimension. 
	\end{rem}
	\section{Acknowledgement}
	HY was financially supported by the University of St Andrews, the University of Cambridge and the Corpus Christi College, Cambridge. HY has received funding from the European Research Council (ERC) under the European Union's Horizon 2020 research and innovation programme (grant agreement No. 803711). The proof of a weaker version of the main result was much more complicated in a previous version of this manuscript. Thanks to an anonymous referee, almost ten pages of technical proofs now turn into this simple-to-state-and-to-prove-and-stronger Theorem \ref{WU}. The author thanks De-Jun Feng, J. Fraser, T. Keleti for fruitful discussions.

	\providecommand{\bysame}{\leavevmode\hbox to3em{\hrulefill}\thinspace}
	\providecommand{\MR}{\relax\ifhmode\unskip\space\fi MR }
	\providecommand{\MRhref}[2]{%
		\href{http://www.ams.org/mathscinet-getitem?mr=#1}{#2}
	}
	\providecommand{\href}[2]{#2}


\begin{thebibliography}{dABCFS11}
		\bibitem[27BMO]{MO} 27-th Brazilian Mathematical Olympiad, third round, problem 6, 2005.
		%\bibitem[A18]{A18}T. Austin. \emph{Measure concentration and the weak Pinsker property}, Publications math\'{e}matiques de l'IH\'{E}S, (2018), 1-119.
		
		\bibitem[D11]{D11} T. Downarowicz. \emph{Entropy in Dynamical Systems}, Cambridge University Press, (2011).
		
		\bibitem[EW11]{EW} M. Einsiedler and T. Ward. \emph{ Ergodic theory: with a view towards number theory}, Graduate Texts in Mathematics, Springer-Verlag London, (2011).
		
		\bibitem[F05]{Fa}K. Falconer. \emph{Fractal geometry: Mathematical foundations and applications, second edition}, John Wiley and Sons, Ltd, 2005.
		\bibitem[F67]{F67} H. Furstenberg. \emph{Disjointness in Ergodic Theory, Minimal Sets, and a Problem in Diophantine Approximation}, Mathematical systems theory, \textbf{1}(1),(1967), 1-49.
		
		
		\bibitem[FX18]{FX18} D-J. Feng and Y. Xiong. \emph{Affine embeddings of Cantor sets and dimension of $\alpha\beta$-sets}, Israel J. Math., 226(2):805--826, 2018.
		
		
		%\bibitem[HS12]{HS12}M. Hochman and P. Shmerkin. \emph{Local entropy averages and projections of fractal measures},Annals of Mathematics, \textbf{175},(2012), 1001-1059.
		
		\bibitem[K79]{K79} Y. Katznelson. \emph{On $\alpha\beta$-sets}, Israel J. Math., 33(1), 1979, 1--4.
		
		
		\bibitem[M99]{Ma1} P. Mattila. \emph{Geometry of sets and measures in Euclidean spaces: Fractals and rectifiability},    Cambridge Studies in Advanced Mathematics, Cambridge University Press, 1999.
		
		\bibitem[OW75]{OW75} D. Ornstein and B. Weiss. \emph{Unilateral codings of Bernoulli systems}, Israel J. Math., 21, (1975), 159--166.
		
		\bibitem[PY98]{PY} M. Pollicott and M. Yuri. \emph{Dynamical systems and ergodic theory}, London Mathematical Society Student Texts, Cambridge University Press, (1998).
		%\bibitem[S12]{S12}D. Simmons. \emph{Conditional measures and conditional expectation; Rohlin's disintegration theorem}, Discrete and continuous dynamical systems, \textbf{32}(7), (2012), 2565-2582.
		%\bibitem[SS05]{SS}E. Stein and R. Shakarchi. \emph{Real analysis: Measure theory, integration and Hilbert spaces},Princeton University Press, (2005).
		
		\bibitem[W16]{Wu} M. Wu, \emph{A proof of Furstenberg's conjecture on the intersections of $\times p$ and $\times q$-invariant sets}, Ann. of Math. (2), \textbf{189}(3), 707-751, (2019).
		
		%\bibitem[Y18a]{Y19}H. Yu. \emph{Discrepancies of irrational rotations, binary expansions of powers of 3 and an improvement on Furstenberg's slicing problem}, preprint, arxiv:1811.11073, (2018).
		
		\bibitem[Y18]{Y18}H. Yu. \emph{Multi-rotations on the unit circle}, J. Number Theory, \textbf{200}, 316-328, (2019).
	\end{thebibliography}
\end{document}